\documentclass[11pt]{amsart}
\usepackage{pxfonts}
\usepackage{amsmath}
\usepackage{amssymb}
\usepackage{amsthm}
\usepackage{amscd, amsfonts, mathrsfs}
\usepackage[dvips]{epsfig}
\usepackage{verbatim}

\usepackage{epsf}
\usepackage[usenames]{color}
\usepackage[bookmarksnumbered,pdfpagelabels=true,plainpages=false,colorlinks=true,
            linkcolor=black,citecolor=black,urlcolor=blue]{hyperref}

\newcommand{\al}{\alpha}

\newcommand{\ep}{\epsilon}

\newcommand{\pbp}{\partial \bar{\partial}}

\newcommand{\vol}{\operatorname{vol}}

\theoremstyle{plain}
\newtheorem{theorem}{Theorem}[section]
\newtheorem{lemma}[theorem]{Lemma}

\theoremstyle{definition}
\newtheorem{rema}[theorem]{Remark}
\newtheorem{defi}[theorem]{Definition}

\numberwithin{equation}{section}
\title[Weighted Interpolation from certain singular hypersurfaces]{Weighted Interpolation from certain singular affine hypersurfaces}

\author[Pingali]{Vamsi P. Pingali}
\address{Department of Mathematics\\
412 Krieger Hall, Johns Hopkins University\\ Baltimore, MD 21218, USA}
\email{vpingali@math.jhu.edu}

\begin{document}
\maketitle

\begin{abstract}
We prove that square integrable holomorphic functions (with respect to a plurisubharmonic weight) can be extended in a square integrable manner from certain singular hypersurfaces (which include uniformly flat, normal crossing divisors) to entire functions in affine space. This gives evidence to a conjecture \cite{PV} regarding the positivity of the curvature of the weight under consideration.
\end{abstract}

\section{Introduction}
\indent Given a countable set of points $\Gamma \subset\mathbb{C}$, one can ask whether it is a set of interpolation for the Bargmann Fock space, i.e., whether for all collections of complex numbers $\{ f_{\gamma}\ ;\ \gamma \in \Gamma\}$ such that 
\[
\displaystyle \sum _{\gamma} \vert f_{\gamma} \vert^2 e^{-\vert \gamma \vert^2} <\infty
\]
one can find  a holomorphic function $F$ such that 
\[
F(\gamma) = f_{\gamma} \text{ for all }\gamma \in \Gamma \quad \text{and} \quad \displaystyle \int_{\mathbb{C}} \vert F \vert^2 e^{-\vert z \vert^2} d\mu_{Leb} <\infty.
\]
 More generally, one may ask the same question with the weight $e^{-\vert z \vert^2}$ replaced by $e^{-\phi}$ for some locally integrable function $\phi$.  If we impose that $\phi$ is $\mathscr{C}^2$-smooth and satisfies $C^{-1} \leq \Delta \phi \leq C$, then a result due to Seip in the classical Bargmann-Fock setting, and to a number of authors in general, states that $\Gamma$ is a set of interpolation for the generalized Bargmann-Fock space if and only if 
\begin{enumerate}
\item $\Gamma$ is uniformly separated with respect to the Euclidean metric, and 
\item the so-called upper density 
\[
D^+_{\phi}(\Gamma) := \limsup _{r \to \infty} \sup _{z \in \mathbb{C}} \frac{\# (\Gamma \cap D(z,r))}{\int _{D(z,r)} \Delta \phi d\mu _{\text{Leb}}}
\]
is strictly less than $1$.
\end{enumerate}
 (See \cite{se} and the references therein). 
 
Higher dimensional generalisations of this problem are of active interest and current research.  \cite{osv, bt, ber, dem, PV}. In higher dimensions, the problem is characterize those hypersurfaces $W \subset \mathbb{C}^n$ with the property that given any holomorphic function $f :W \to \mathbb{C}$ satisfying $\displaystyle \int _{W} \vert f \vert^2 e^{-\phi} < \infty$, there is a holomorphic function $F$ that extends $f$ and satisfies $\displaystyle \int _{\mathbb{C}^n} \vert F \vert^2 e^{-\phi} < \infty$.  If we assume that $\phi$ is smooth and satisfies  
\begin{equation}\label{ddbar-bound}
C^{-1}\sqrt{-1}\pbp|z|^2  \leq \sqrt{-1}\pbp \phi \leq C \sqrt{-1}\pbp |z|^2,
\end{equation}
then sufficient conditions exist:  In \cite{osv} sufficient conditions were provided to solve this problem for smooth $W$.  Akin to ``uniform separation", a geometric notion called ``uniform flatness" was defined, as was a corresponding generalisation of the notion of ``upper density".  It was then proved that if $W$ is uniformly flat and has density strictly less than $1$, then $W$ is an interpolation hypersurface.  The reason for the lower bounds on $\pbp \phi$ is that the desired interpolating function was constructed by ``patching up" local extensions using the H\"ormander theorem (which in turn  requires strict positivity of the curvatures involved).   The upper bound on $\pbp \phi$ is related to the approach taken in \cite{osv} to extend the data from $W$ to a small neighbourhood.  In \cite{PV}, the same result for smooth hypersurfaces $W$ was established for any plurisubharmonic weight $\phi$. The main tool used was an Ohsawa-Takegoshi type extension theorem (theorem \ref{ot-thm}).   As far as the necessity, little is known.  The necessity of the density condition is wide open, but it was also shown in \cite{PV} that uniform flatness is not a necessary condition as soon as $n \ge 2$.

In \cite{PV} the problem of weighted interpolation for singular $W$ was also studied. A corresponding notion of uniform flatness was defined for singular varieties and was proven to be one of the sufficient conditions (the other one being upper density less than $1$) required for the solution of the interpolation problem provided the weight $\phi$ satisfies \eqref{ddbar-bound}.  Once again the H\"ormander theorem was used in the singular case to patch up local extensions. It was conjectured that the condition on the weight may be weakened. The purpose of this paper is to provide evidence for this conjecture by weakening the strictness of plurisubharmonicity of $\phi$ for certain singular hypersurfaces (that include the case of ``uniformly flat" simple normal crossing divisors). The strategy of proof is to use the Ohsawa-Takegoshi type theorem \ref{ot-thm} (as in \cite{PV} for the smooth case) to solve the problem of weighted interpolation. 

\textbf{Acknowledgements}: The author is greatly indebted to Dror Varolin for fruitful discussions and for a careful reading of the paper.

\section{Statements of results}
\indent Let $\omega_0=\frac{\sqrt{-1}}{2} \sum dz^i \wedge d\bar{z}^i$ be the Euclidean K\"ahler form on $\mathbb{C}^n$. We recall a few definitions before stating our main theorem. \\
\indent Firstly, we define uniform flatness \cite{osv} for smooth hypersurfaces in $\mathbb{C}^n$.
\begin{defi}
A smooth hypersurface $W \subset \mathbb{C}^n$ is said to be uniformly flat if there exists a tubular neighbourhood of $W$ of radius $\epsilon_0$ where $\epsilon_0$ is a positive constant.
\end{defi}
Next we recall the definition of  the same concept in the special case of singular hypersurfaces of the type $W=T^{-1} (0)$ where $T=T_1 T_2 \ldots T_k $ is an entire function such that $W_i = T_i ^{-1} (0)$ is a smooth, uniformly flat hypersurface.
\begin{defi}
A singular hypersurface $W=\displaystyle \cup_{i=1} ^{k} W_i$ where $W_i$ are smooth hypersurfaces is said to be uniformly flat if the $W_i$ are uniformly flat, and there is a positive constant $\epsilon _0$ (which is the radius of the uniform tubular neighbourhood of the $W_k$) so that the angles of intersection $W_k \cap W_l$ lie in $[\epsilon _0, \pi - \epsilon _0]$.
\end{defi}
Finally, we define the concept of density \cite{osv}.
\begin{defi}
Let $W$ be a hypersurface.  The $(1,1)$-form $\Upsilon _r^W = [W]*\frac{\mathbf{1}_{B(0,r)}}{\vol(B(0,r))}$ is called the density tensor of $W$.  If $\phi :\mathbb{C}^n \rightarrow \mathbb{R}$ is a plurisubharmonic function,  then with $\phi_r (z) = \displaystyle \fint _{B(z,r)} \phi (\zeta) \omega _0 ^n (\zeta)$, the number 
\begin{gather}
D^+_{\phi}(W)= \inf \left \{ \alpha  \geq 0 \ ;\  \sqrt{-1} \pbp \phi _r - \tfrac{1}{\alpha}\Upsilon _r^W \ge 0 \text{ for all }r>>0 \right \}
\end{gather}
is called the upper density of $W$.
\end{defi}
\begin{rema}
The condition that the density is less than $1$ is equivalent to saying that there is a positive constant $\delta$ so that $ \sqrt{-1} \pbp \phi _r \geq (1+\delta)\Upsilon _r^W $ for $r >>1$.
\end{rema}
\begin{rema}
The reason the weighted average $\phi _r$ is used (as opposed to $\phi$) is to make the theorem look similar to its one-dimensional version, where the weight $\phi _r$ appears naturally in the proof of necessity of the density.  (It arises there from the use of Jensen's Formula.) Thanks to Lemma 1.1 in \cite{PV} which is in turn taken from \cite{quimbo, li}, it turns out that if $\pbp \phi$ is bounded above by a multiple of the Euclidean metric, we may replace $\phi _r$ by $\phi$ and get equivalent $L^2$ norms. 
\label{qui}
\end{rema}
We are finally in a position to state our main theorem.
\begin{theorem}
Let $\phi : \mathbb{C}^n \rightarrow \mathbb{R}$ be a $\mathcal{C}^2$-smooth plurisubharmonic function satisfying $0\leq \sqrt{-1} \partial \bar{\partial} \phi_r \leq C \omega$ for some positive constant $C$, and let $W \subset \mathbb{C}^n$ be a possibly singular uniformly flat, complex hypersurface such that $W = T^{-1}(0)$ where $T=T_1 T_2 \ldots T_k$ is a holomorphic function with $T_i ^{-1} (0)$ being uniformly flat, smooth complex hypersurfaces such that $D^{+}_{\phi} (W) <1$. Then any holomorphic function $f$ on $W$ such that $\displaystyle \int _W \vert f \vert^2 e^{-\phi}  \omega_0 ^{n-1} < \infty$ can be extended to a holomorphic entire function $F$ satisfying $\displaystyle \int _{\mathbb{C}^n} \vert F \vert^2 e^{-\phi}  \omega_0 ^n \leq K\displaystyle \int _W \vert f \vert^2 e^{-\phi}  \omega_0 ^{n-1} $ for some constant $K>0$ independent of $f$ but possibly dependent on $C$.
\label{main}
\end{theorem}
\begin{rema}
The $\mathcal{C}^2$-smoothness condition of theorem \ref{main} is actually not necessary. Indeed, there is a sequence of plurisubharmonic functions $\phi _{\ep}$ decreasing pointwise to $\phi_r$. There is a corresponding sequence of holomorphic functions $F_{\ep}$ extending $f$ satisfying $\displaystyle \int _{\mathbb{C}^n} \vert F_{\ep} \vert^2 e^{-\phi _{\ep}} \omega _0 ^n \leq K \int _{W} \vert f \vert^2 e^{-\phi _{\ep}} \omega _0 ^{n-1} < K\int _{W} \vert f \vert^2 e^{-\phi} \omega _0 ^{n-1}$. Thus a subsequence $F_{\ep} e^{-\phi _{\ep} /2}$ goes weakly to a function $G$ in $L^2$. This means that a subsequence $F_{\ep}$ converges to some function $F$ almost everywhere and hence weakly in $L^2$. This means that $F$ is holomorphic, extends $f$, and satisfies the desired estimates. 
\label{rem1}
\end{rema}
\begin{rema}
For the purposes of algebraic geometry it is important to let $\phi$ be singular with non-zero Lelong number. Our proof relies strongly on the aforementioned lemma in \cite{quimbo, li} which in turn requires an upper bound on $\pbp \phi _r$. Removing this upper bound seems to require fundamentally new ideas.
\label{rem2}
\end{rema}
\section{Proof of the main theorem}
\indent Let $W_i = T_i ^{-1} (0)$ and $f_i = f\vert_{W_i}$. We extend $f$ using an inductive procedure. Using theorem $1$ in \cite{PV}, and the fact that $W_1$ is a uniformly flat, smooth hypersurface with $D^{+}_{\phi}(W_1) < 1$, we may extend $f_1$ to $F_1$ satisfying $\displaystyle \int _{\mathbb{C}^n} \vert F_1 \vert^2 e^{-\phi} \leq C  \int _{W_1} \vert f_1 \vert^2 e^{-\phi} < \infty$ (where $C>0$ is a constant independent of $f$). Let us assume that $f_1, \ldots, f_i$ have been extended to an entire function $F_i$ satisfying the $L^2$ estimates. Let $\tilde{f}_{i+1} = f  - F_i \vert_{W}$. If we manage to extend $\tilde{f}_{i+1}$ from $W_{1}\cup W_{2} \ldots \cup W_{i+1}$ to $\tilde{F}_{i+1}$ with good estimates, then $F_{i+1} = F_i + \tilde{F}_{i+1}$ will extend $f_1, \ldots, f_{i+1}$. \\
\indent Before we implement this strategy we need to prove a lemma regarding the norm of $F_i \vert _{W}$ 
\begin{lemma}
For any holomorphic function $F$ on $\mathbb{C}^n$ satisfying $\int _{\mathbb{C}^n} \vert F \vert ^2 e^{-\phi} \omega_0 ^n < \infty$ and any smooth, uniformly flat hypersurface $V=\mathcal{T}^{-1}(0)$
\begin{equation}
\displaystyle \int _{V} \vert F \vert^2 e^{-\phi} \omega _0 ^{n-1} \leq C \int _{\mathbb{C}^n} \vert F \vert ^2 e^{-\phi} \omega_0 ^n 
\end{equation}
for some positive constant $C$ which is independent of $F$.
\label{loop}
\end{lemma}
\begin{proof}
By uniform flatness \cite{osv} there exists an open cover of $\mathbb{C}^n$ by balls $B_{\al}$ of some positive radius lying in  $[a,2a]$ such that in $B_{\al}\cap V$, the hypersurface $V$ is locally a graph $y=g_u(x)$ over a small disc in the tangent space of any point $z$ in $B_{\al}$ such that $\vert g_u(x) \vert \leq C_a \vert x \vert^2$ where $C_a$ is independent of $z$. Moreover every point in $\mathbb{C}^n$ is contained in at most $N$ balls and similar properties hold for concentric balls with the double the radius of $B_{\al}$ (which we denote by $\overline{B}_{\al}$). Let $\rho_{\al}$ be a partition of unity subordinate to the open cover defined by the $B_{\al}$. \\
\indent Hence $\displaystyle \int _{V} \vert F \vert^2 e^{-\phi} \omega _0 ^{n-1} = \displaystyle \sum _{\al} \int _{V\cap B_{\al}} \rho _{\al} \vert F \vert^2 e^{-\phi} \omega _0 ^{n-1}$. If we manage to prove that $\displaystyle\int _{V\cap B_{\al}} \vert F \vert^2 e^{-\phi} \omega _0 ^{n-1} \leq \int _{\overline{B_{\al}}} \vert F \vert^2 e^{-\phi} \omega _0 ^{n-1}$ for all $\al$, then we will be done. \\
\indent Fix a $B_{\al}$ (whose radius is $r_{\al}$). Using uniform flatness we may assume without loss of generality that $V\cap B_{\al}$ is actually $(z_1 = 0) \cap B_{\al}$. Fix a point $z=(0,z_2,z_3,\ldots,z_n)$ in $V\cap B_{\al}$. Using the lemma from \cite{quimbo, li} mentioned earlier and uniform flatness we may assume without loss of generality that there exists a holomorphic function $H$ on $B_{\al}$ such that $e^{-\phi}(z) = \vert H(z) \vert^2$ and $\Vert e^{-\phi}/\vert H \vert^2 \Vert _{C^0 (B_{\al})} \leq C$ where $C$ is independent of the point $z$ (but $H$ can potentially depend on $z$). Thus
\begin{gather}
\displaystyle \vert F\vert ^2 e^{-\phi} (0,z_2,\ldots,z_n) = \vert F H\vert^2 (0,z_2,\ldots,z_n) \leq C \int _{\vert z_1 \vert < r_{\al}}  \vert F H\vert^2 (z_1,z_2,\ldots,z_n) \frac{\sqrt{-1}}{2} dz_1 \wedge d\bar{z}_1 \nonumber \\
\leq C \int _{\vert z_1 \vert < r_{\al}}  \vert F \vert^2 e^{-\phi} (z_1,z_2,\ldots,z_n) \frac{\sqrt{-1}}{2} dz_1 \wedge d\bar{z}_1 \nonumber
\end{gather}
Integrating over all $(0,z_2,\ldots,z_n)$ in $V\cap B_{\al}$ we see that 
\begin{gather}
\int _{V\cap B_{\al}} \vert F \vert^2 e^{-\phi} \omega _0 ^{n-1} \leq \int _{\overline{B_{\al}}} \vert F \vert^2 e^{-\phi} \omega _0 ^{n-1} \nonumber
\end{gather}
As mentioned earlier this is enough to finish the proof. 
\end{proof}
Lemma \ref{loop} implies that $\displaystyle \int _{W} \vert F_i \vert ^2 e^{-\phi} \omega_0 ^{n-1} = \sum _{j} \displaystyle \int _{W_j} \vert F_i \vert ^2 e^{-\phi} \omega_0 ^{n-1} \leq C \int _{\mathbb{C}^n} \vert F_i \vert ^2 e^{-\phi} \omega_0 ^{n}$. \\
\indent To extend $\tilde{f}_{i+1} = f-F_i \vert_{W_i}$ we use the following $L^2$-extension theorem from \cite{v-tak} :
\begin{theorem}\label{ot-thm}
Let $X$ be a Stein manifold with K\"ahler form $\omega$, $Z \subset X$ a smooth hypersurface, $e^{-\eta}$ a singular Hermitian metric for the holomorphic line bundle associated to the smooth divisor $Z$, and $T$ a holomorphic section of this line bundle such that $Z= \{T=0\}$.  Assume that $e^{-\eta}|_Z$ is still a singular Hermitian metric, and that
\[
\sup _X |T|^2 e^{-\eta} = 1.
\]
Let $H \to X$ be a holomorphic line bundle with singular Hermitian metric $e^{-\kappa}$ whose curvature $\sqrt{-1} \partial \bar{\partial} \kappa$ is non-negative in the sense of currents.  Suppose also that
\[
\sqrt{-1} \partial \bar{\partial} \kappa + {\rm Ricci}(\omega) \ge (1+\delta) \sqrt{-1} \partial \bar{\partial} \eta
\]
for some positive number $\delta$.  Then for each section $f \in H^0(Z, H)$ satisfying
\[
\int _{Z} \frac{|f|^2 e^{-\kappa}}{|dT|^2 e^{-\eta}} \frac{\omega ^{n-1}}{(n-1)!} < +\infty
\]
there is is a section $F \in H^0(X,H)$ such that
\[
F|_Z = f \quad \text{and}\quad \int _{X} |F|^2 e^{-\kappa} \frac{\omega ^n}{n!} \le \frac{C}{\delta} \int _Z  \frac{|f|^2 e^{-\kappa}}{|dT|^2 e^{-\eta}} \frac{\omega ^{n-1}}{(n-1)!},
\]
where the constant $C$ is universal.
\end{theorem}
%Choosing $X=\mathbb{C}^n$, $\omega= \omega_0$, $T$ is an entire function, $\eta (x) = \fint _{B(x,r)} \ln (\vert T \vert ^2) \frac{\omega_0 ^n}{n!}$, and $H$ to be trivial with a metric $e^{-\kappa} = e^{-\phi}$, we arrive the following theorem.
%\begin{theorem}
%Let $W=T^{-1}(0)$ be a uniformly flat (possibly singular) hypersurface. If $\phi$ is a $\mathcal{C}^2$-smooth plurisubharmonic function so that  $D^{+}_{\phi}(W) <\frac{1}{1+\delta}$ for a positive constant $\delta$, then for each holomorphic function $f$ defined on $W$ satisfying $\displaystyle \int _W \frac{\vert f \vert^2 e^{-\phi}}{\vert dT \vert^2 e^{-\fint _{B(x,r)} \ln (\vert T \vert ^2) \frac{\omega_0 ^n}{n!}}}\omega_0 ^{n-1} < \infty$, there is an extension $F$ on $\mathbb{C}^n$ satisfying
%$$\displaystyle \vert F \vert^2 e^{-\phi} \frac{\omega_0 ^n}{n!} \leq \frac{C}{\delta} \int _W  \frac{\vert f \vert^2 e^{-\phi}}{\vert dT \vert^2 e^{-\fint _{B(x,r)} \ln (\vert T \vert ^2) \frac{\omega_0 ^n}{n!}}}\frac{\omega_0 ^{n-1}}{(n-1)!}$$
%\label{otform}
%\end{theorem}
By lemma $4.15$ of \cite{PV} we see that $\tilde{f}_{i+1}$ is divisible by $S_i= T_1 T_2 \ldots T_i$. The following lemma is crucial.
\begin{lemma}
\begin{gather}
\displaystyle\int _{W_{i+1}} \frac{\left \vert \frac{\tilde{f}_{i+1}}{S_i}  \right \vert^2 e^{-\phi}}{\vert dT _{i+1} \vert^2 e^{-\fint _{B(z,r)} \ln (\vert T_1 \ldots T_{i+1} \vert ^2)}}\omega_0 ^{n-1} \leq C \displaystyle\int _{W_{i+1}} \left \vert \tilde{f}_{i+1}  \right \vert^2 e^{-\phi}\omega_0 ^{n-1}\nonumber
\end{gather}
where $C>0$ is a constant independent of $\tilde{f}_{i+1}$.
\label{est}
\end{lemma}
\begin{proof}
The integral may be written as
\begin{gather}
\displaystyle\int _{W_{i+1}} \frac{\left \vert \tilde{f}_{i+1}  \right \vert^2 e^{-\phi}}{\vert S_i \vert ^2 e^{-\fint _{B(z,r)} \ln (\vert S_i \vert ^2)} \vert dT _{i+1} \vert^2 e^{-\fint _{B(z,r)} \ln (\vert T_{i+1} \vert ^2)}}\omega_0 ^{n-1} \nonumber
\end{gather}
where we recall that $S_i = T_1 \ldots T_i$. The function $\vert dT _{i+1} \vert^2 e^{-\fint _{B(z,r)} \ln (\vert T_{i+1} \vert ^2)}$ is bounded below by virtue of uniform flatness (lemma $4.11$ in \cite{PV}). The expression $\vert S_i \vert ^2 e^{-\fint _{B(z,r)} \ln (\vert S_i \vert ^2)}$ may be written as
\begin{gather}
\displaystyle \prod _{u=1} ^{u=i}\left ( \vert T_u \vert ^2 e^{-\fint _{B(z,a)} \ln (\vert T_u \vert ^2)} \right )e^{\fint _{B(z,a)} \ln (\vert S_i \vert ^2)-\fint _{B(x,r)} \ln (\vert S_i \vert ^2)} \nonumber
\end{gather}
 where $a$ is chosen to be so small that $W_u ^{-1} (0) \ \forall \ u=1,2,\ldots i$ is (as earlier) a graph $y=g_u(x)$ over a small disc (of radius $>a$) in the tangent space of $z$ such that $\vert g_u (x)\vert \leq  C_a \vert x \vert^2$ where $C_a$ is independent of $z$ (once again by uniform flatness \cite{osv}). 

Since expressions of the form $\vert T\vert ^2 e^{-\fint _{B(z,a)} \ln (\vert T\vert ^2) d\mu(\tilde{z})}$ are independent of the choice of defining function $T$, we may replace $T_u$ in the expression above with $y(\tilde{z})-g_u(x(\tilde{z}))$.( Notice that the $\tilde{z}$-coordinates may be changed to the $(y,x)$-coordinates by means of an affine isometry.) It is easy to see  that $e^{-\fint _{B(z,a)} \ln (\vert y - g_u(x)\vert ^2)}$ is uniformly bounded below independent of $z$. So we infer that
\begin{gather}
\displaystyle\int _{W_{i+1}} \frac{\left \vert \tilde{f}_{i+1}  \right \vert^2 e^{-\phi}}{\vert S_i \vert ^2 e^{-\fint _{B(z,r)} \ln (\vert S_i \vert ^2)} \vert dT _{i+1} \vert^2 e^{-\fint _{B(z,r)} \ln (\vert T_{i+1} \vert ^2)}}\omega_0 ^{n-1} \nonumber \\ \leq  C\displaystyle\int _{W_{i+1}} \frac{\left \vert \tilde{f}_{i+1}  \right \vert^2 e^{-\phi}}{\displaystyle \prod _{u=1} ^{u=i}\left ( \vert y(z) - g_u (x(z)) \vert ^2 \right ) e^{\fint _{B(z,a)} \ln (\vert S_i \vert ^2)-\fint _{B(z,r)} \ln (\vert S_i \vert ^2)}}\omega_0 ^{n-1}
\label{sofarone}
\end{gather}
The factor $\displaystyle e^{\fint _{B(z,a)} \ln (\vert S_i \vert ^2)-\fint _{B(z,r)} \ln (\vert S_i \vert ^2)}$ is seen to be uniformly bounded below using the proof of lemma $4.11$ in \cite{PV}. So far we have
\begin{gather}
\displaystyle\int _{W_{i+1}} \frac{\left \vert \frac{\tilde{f}_{i+1}}{S_i}  \right \vert^2 e^{-\phi}}{\vert dT _{i+1} \vert^2 e^{-\fint _{B(z,r)} \ln (\vert T_1 \ldots T_{i+1} \vert ^2)}}\omega_0 ^{n-1} \leq C \displaystyle\int _{W_{i+1}} \frac{\left \vert \tilde{f}_{i+1}  \right \vert^2 e^{-\phi}}{\displaystyle \prod _{u=1} ^{u=i}\left ( \vert y(z) - g_u (x(z)) \vert ^2 \right )}\omega_0 ^{n-1}
\label{sofartwo}
\end{gather}
\indent At this juncture choose an open cover of $W_{i+1}$ by balls $B_{\alpha}$ in $\mathbb{C}^n$ of a sufficiently small radius $\frac{\epsilon_0}{2} \leq b_{\alpha} \leq (k+1)\epsilon _0$ where $k$ is a constant. In fact, by virtue of uniform flatness one may choose these balls so that uniform paracompactness holds \cite{PV}. Extend this to an open cover of $\mathbb{C}^n$ with the same properties. Choose a partition of unity $\rho_{\al}$ subordinate to this open cover. Inequality \ref{sofartwo} may be written as
\begin{gather}
\displaystyle\int _{W_{i+1}} \frac{\left \vert \frac{\tilde{f}_{i+1}}{S_i}  \right \vert^2 e^{-\phi}}{\vert dT _{i+1} \vert^2 e^{-\fint _{B(z,r)} \ln (\vert T_1 \ldots T_{i+1} \vert ^2)}}\omega_0 ^{n-1} \leq C \displaystyle \sum _{\al} \int _{W_{i+1}\cap B_{\alpha}} \rho_{\al}\frac{\left \vert \tilde{f}_{i+1}  \right \vert^2 e^{-\phi}}{\displaystyle \prod _{u=1} ^{u=i}\left ( \vert y(z) - g_u (x(z)) \vert ^2 \right )}\omega_0 ^{n-1}
\label{sofarthree}
\end{gather}
Using lemma $4.16$ of \cite{PV} we see that, at the cost of increasing the radius of the balls $B_{\alpha}$ by a factor to get new, bigger balls $\hat{B}_{\alpha}$, we have
\begin{gather}
\displaystyle\int _{W_{i+1}} \frac{\left \vert \frac{\tilde{f}_{i+1}}{S_i}  \right \vert^2 e^{-\phi}}{\vert dT _{i+1} \vert^2 e^{-\fint _{B(z,r)} \ln (\vert T_1 \ldots T_{i+1} \vert ^2)}}\omega_0 ^{n-1} \leq C \displaystyle \sum _{\al} \int _{W_{i+1}\cap \hat{B}_{\alpha}} \left \vert \tilde{f}_{i+1}  \right \vert^2 e^{-\phi}\omega_0 ^{n-1}
\label{sofarthree}
\end{gather}
Since each point of $\mathbb{C}^n$ is contained in a finite number of balls (let us say $N$),
\begin{gather}
\displaystyle\int _{W_{i+1}} \frac{\left \vert \frac{\tilde{f}_{i+1}}{S_i}  \right \vert^2 e^{-\phi}}{\vert dT _{i+1} \vert^2 e^{-\fint _{B(z,r)} \ln (\vert T_1 \ldots T_{i+1} \vert ^2)}}\omega_0 ^{n-1} \leq CN \int _{W_{i+1}} \left \vert \tilde{f}_{i+1}  \right \vert^2 e^{-\phi}\omega_0 ^{n-1}
\label{sofarfour}
\end{gather}
\end{proof}

\emph{Proof of theorem \ref{main}}: We now extend $\tilde{f}_{i+1}$ from $Z_{i+1}=W_{1}\cup W_2 \ldots \cup W_{i+1}$ to an entire function $\tilde{F}_{i+1}$ in $\mathbb{C}^n$ using theorem \ref{ot-thm} by choosing $X$ to be $\mathbb{C}^n$, $Z=Z_{i+1}$, $\kappa = \phi_r$ (with the understanding that according to remark \ref{qui} the norms defined by $\phi_r$ and $\phi$ are equivalent), $H$ to be trivial,  and $\eta = \fint _{B(z,r)} \ln (\vert T_1 \ldots T_{i+1} \vert ^2) $. The density condition ensures that the curvature condition in theorem \ref{ot-thm} is satisfied. The only hypotheses that of theorem \ref{ot-thm} that needs to be checked now is the finiteness of $\int _Z  \frac{|f|^2 e^{-\kappa}}{|dT|^2 e^{-\eta}} \frac{\omega ^{n-1}}{(n-1)!}$. Indeed in our case,
\begin{gather}
\int _Z  \frac{|f|^2 e^{-\kappa}}{|dT|^2 e^{-\eta}} \omega ^{n-1} = \displaystyle \sum _{j=1} ^{i+1}  \int _{W_{j}} \frac{\left \vert \tilde{f}_{i+1}  \right \vert^2 e^{-\phi_r}}{\vert dT \vert^2 e^{-\fint _{B(z,r)} \ln (\vert T_1 \ldots T_{i+1} \vert ^2)}}\omega_0 ^{n-1} \nonumber 
= \int _{W_{i+1}} \frac{\left \vert \frac{\tilde{f}_{i+1}}{S_i}  \right \vert^2 e^{-\phi_r}}{\vert dT _{i+1} \vert^2 e^{-\fint _{B(z,r)} \ln (\vert T_1 \ldots T_{i+1} \vert ^2)}}\omega_0 ^{n-1} \nonumber 
\end{gather}
which by lemma \ref{est} and remark \ref{qui} is bounded above by $C\displaystyle \int _{W_{i+1}}\vert \tilde{f} _{i+1} \vert^2 e^{-\phi} \omega_{0} ^{n-1}\leq C \int _{W}(\vert f \vert^2 + \vert  F_i \vert^2)e^{-\phi} \omega_{0} ^{n-1}$ for some constant $C>0$. \\
\indent By lemma \ref{loop} we may conclude that $\int _{W}\vert  F_i \vert^2e^{-\phi} \omega_{0} ^{n-1} \leq C \int _{\mathbb{C}^n}\vert F_i \vert^2 e^{-\phi} \omega_{0} ^{n}$ which is in turn less than (by the inductive hypothesis) $C\int _{\mathbb{C}^n}\vert f \vert^2 e^{-\phi} \omega_{0} ^{n}$. \\ 
\indent As mentioned earlier, the definition $F_{i+1} = F_{i} + \tilde{F}_{i+1}$ completes the inductive step. Hence we are done. \qed

\end{document}